\documentclass[a4paper]{amsart}

\def\gam{\gamma }
\def\RR{\mathbb R}

\newcommand{\set}[1]{\{#1\}}
\providecommand{\abs}[1]{\lvert#1\rvert}
\providecommand{\norm}[1]{\lVert#1\rVert}
\newcommand{\remove}[1]{ }

\newtheorem{theorem}{Theorem}[section]

\newtheorem{lemma}[theorem]{Lemma}

\theoremstyle{definition}

\theoremstyle{remark}

\newtheorem*{remarks}{Remarks}

\numberwithin{equation}{section}

\begin{document}
\title[Ingham--Beurling type estimates]{Explicit multidimensional Ingham--Beurling type estimates}
\author{Vilmos Komornik}
\address{Delft University of Technology, Mekelweg 4, 2628 CD Delft, The Netherlands}
\date{December 1st, 2008}

\begin{abstract}
Recently a new proof was given for Beurling's Ingham type theorem on one-dimensional nonharmonic Fourier series, providing explicit constants.
We improve this result by applying a short elementary method instead of the previous complex analytical approach. 
Our proof equally works in the multidimensional case.
\end{abstract}

\maketitle

\section{Introduction}\label{s1}

Let $(\omega_k)_{k\in K}$ be a family of vectors in $\RR^N$ satisfying the gap condition
\begin{equation*}
\gam:=\inf_{k\ne n}\abs{\omega_k-\omega_n}>0.
\end{equation*}
Let $K=K_1\cup \cdots \cup K_m$ be a finite partition of $K$ and set
\begin{equation*}
\gam_j:=\inf\set{\abs{\omega_k-\omega_n}\ :\ k,n\in K_j\quad\text{and}\quad k\ne n},\quad j=1,\ldots, m.
\end{equation*}
We denote by $B_R$ the open ball of radius $R$ in $\RR^N$, centered at the origin, and by $\mu$ the first eigenvalue of $-\Delta$ in the Sobolev space $H_0^1(B_1)$. Finally, we set
\begin{equation*}
R_j:=\frac{2\sqrt{\mu}}{\gam_j}
\quad\text{and}\quad
R_0:=R_1+\cdots+R_m.
\end{equation*}

We prove the following theorem:

\begin{theorem}\label{t11}
There exist two positive constants $c_1$ and $c_2$, depending only on $N$, $\gam$, $m$ and $\gam_1$,\ldots, $\gam_m$, such that if $R_0<R\le 2R_0$, then
\begin{equation}\label{11}
c_1 (R-R_0)^{5m-4+2N}\sum_{k\in K}\abs{x_k}^2
\le \int_{B_R}\left\vert \sum_{k\in K} x_k e^{i\omega_k\cdot t}\right\vert^2\ dt
\le c_2 \sum_{k\in K}\abs{x_k}^2
\end{equation}
for every square summable family $(x_k)_{k\in K}$ of complex numbers.
\end{theorem}

\begin{remarks}\mbox{}

\begin{itemize}
\item The estimates remain valid for all balls of radius $R$ by translation invariance. The explicit values of $c_1$ and $c_2$ can be computed easily from the proof below.

\item The proof remains valid if in the multidimensional case we replace $\abs{\omega_k-\omega_n}$ in the definition of $\gam$ and $\gam_j$ by the $L^p$ norm $\norm{\omega_k-\omega_n}_p$ for some $1\le p\le \infty$ and we replace $\mu$ in the statement of the theorem by the first eigenvalue of $-\Delta$ in the Sobolev space $H_0^1(B_1^p)$ with
    \begin{equation*}
    \set{x\in\RR^N\ :\ \norm{x}_p<1}.
    \end{equation*}

\item In the one-dimensional case we may assume that $(\omega_k)$ is a (doubly infinite) increasing sequence satisfying
\begin{equation*}
\gam:=\inf (\omega_{k+1}-\omega_k)>0.
\end{equation*}
Putting
\begin{equation*}
\gam_m':=\inf \frac{\omega_{k+m}-\omega_k}{m}
\end{equation*}
for some $m$, we deduce from the theorem that the estimates \eqref{11} hold on the interval $(-R,R)$ if $R_0:=2\pi/\gam_m'$ and $R_0<R\le 2R_0$.
This definition of $\gam_m'$ was introduced in \cite{BaiKomLor2000}, \cite{BaiKomLor2002} where it was also shown that the Ingham type condition $2\pi R/\gam_m'$ is equivalent to Beurling's original condition based on the P\'olya upper density. Our paper thus also provides a short elementary proof of Beurling's theorem.

\item The case $N=1$ of our result is stronger than a theorem recently proved in \cite{TenTuc2007} by employing deeper tools of complex analysis. Contrary to a claim in \cite{TenTuc2007} we show that an elementary approach used in \cite{Kom1992} leads to explicit constants, even in higher dimensions. In order to get better constants $c_1$ and $c_2$ (this is not necessary in order to get the exponent $5m-4+2N$) we use a new proof of the multidimensional Ingham type theorem given in \cite{BaiKomLor1998}, \cite{BaiKomLor1999}.
\end{itemize}
\end{remarks}

\section{Proof of the theorem}
\label{s2}

We introduce some notations. We fix an eigenfunction $H$ of $-\Delta$ in $H_0^1(B_1)$ corresponding to the first eigenvalue $\mu$. We may assume that $H>0$ in $B_1$. The function $H$ is even and we  extend it by zero to the whole $\RR^N$. We denote by $h$ the Fourier transform of $H$ defined by
\begin{equation*}
h(t):=\int_{B_1}H(x)e^{-ix\cdot t}\ dx=\int_{B_1}H(x)\cos (x\cdot t)\ dx;
\end{equation*}
we observe that
\begin{equation*}
\min_{\abs{t}\le\pi/2}\abs{h(t)}^2>0.
\end{equation*}

Putting $H_s(x):=H(s^{-1}x)$ where $s$ is a given positive number, a scaling argument shows that $H_s$ is the first eigenfunction of $-\Delta$ in $H_0^1(B_s)$ with the corresponding
eigenvalue $\mu_s=s^{-2}\mu$. Furthermore, we have
\begin{equation*}
\norm{H_s}_{L^2(B_s)}^2=s^N\norm{H}_{L^2(B_1)}^2
\end{equation*}
for all $1\le p<\infty$, and the Fourier transform $h_s$ of $H_s$ is given by the formula $h_s(t)=s^Nh(st)$.

We set
\begin{equation*}
r:=\frac{R-R_0}{2m}
\end{equation*}
for brevity, so that $0<r\le R_0/(2m)$ by the assumptions of the theorem.

The following two lemmas are due to Ingham \cite{Ing1936} in one dimension and to Kahane \cite{Kah1962} in several dimensions. We recall their simple proof given in \cite{BaiKomLor1998}, \cite{BaiKomLor1999}  in order to precise the nature of the constants appearing in the estimates. Here and in the sequel the letters $\alpha$, $\alpha_i$ stand for diverse positive constants depending only on $N$, $\gam$, $m$ and $\gam_m$, and the value of $\alpha$ may be different for different occurrences.

\begin{lemma}\label{l21}
We have
\begin{equation*}
\int_{B_R}\left\vert \sum_{k\in K} x_k e^{i\omega_k\cdot t}\right\vert^2\ dt\le \alpha_0 \sum_{k\in K}\abs{x_k}^2.
\end{equation*}
\end{lemma}

\begin{proof}
Setting $G:=H_{\gam/2}*H_{\gam/2}$ and denoting by $g$ its Fourier transform we have $g=\abs{h_{\gam/2}}^2\ge 0$. Writing
\begin{equation*}
x(t):=\sum_{k\in K} x_k e^{i\omega_k\cdot t}
\end{equation*}
for brevity and applying the Fourier inversion formula we obtain that
\begin{multline*}
\left( \min_{B_{\pi/\gam}}g\right) \int_{B_{\pi/\gam}} \abs{x(t)}^2\ dt\le\int_{\RR^N}g(t)\abs{x(t)}^2\ dt\\
=(2\pi)^N\sum_{k,n\in K}G(\omega_k-\omega_n)x_k\overline{x_n}
=(2\pi)^NG(0)\sum_{k\in K}\abs{x_k}^2
\end{multline*}
because for $k\ne n$ the vector $\omega_k-\omega_n$ lies outside the support of $G$ by our gap assumption.

This proves the lemma for $B_{\pi/\gam}$ instead of $B_R$ with
\begin{equation*}
\alpha_0'=\frac{(2\pi)^NG(0)}{\min_{B_{\pi/\gam}}g}\remove{=\frac{(2\pi)^N\norm{H_{\gam/2}}_2^2}{\min_{\abs{t}\le\pi/\gam}\abs{h_{\gam/2}(t)}^2}}
\end{equation*}
in place of $\alpha_0$. Since $B_R$  may be covered by at most $(1+{R\gam}/{\pi}) ^N$ translates of $B_{\pi/\gam}$ and $R\le 2R_0$,
the lemma follows with
\begin{equation*}
\alpha_0=\left(1+\frac{2R_0\gam}{\pi}\right)^N\alpha_0'.\qedhere
\end{equation*}
\end{proof}

\begin{lemma}\label{l22}
We have
\begin{equation*}
\alpha_j r\sum_{k\in K_j}\abs{x_k}^2\le \int_{B_{R_j+r}}\left\vert \sum_{k\in K_j} x_k e^{i\omega_k\cdot t}\right\vert^2\ dt,\quad j=1,\ldots, m.
\end{equation*}
\end{lemma}

\begin{proof}
Setting $G:=\bigl[(R_j+r)^2+\Delta\bigr](H_{\gam_j/2}*H_{\gam_j/2})$ and denoting its Fourier transform by $g$ we have
\begin{equation*}
g(t)=\bigl[(R_j+r)^2-\abs{t}^2\bigr]\abs{h_{\gam_j/2}(t)}^2.
\end{equation*}
We have $g\le 0$ outside $B_{R_j+r}$ and $g\le \alpha$ in $B_{R_j+r}$, so that
writing
\begin{equation*}
x_j(t):=\sum_{k\in K_j} x_k e^{i\omega_k\cdot t}
\end{equation*}
and applying the Fourier inversion formula we obtain that
\begin{equation*}
(2\pi)^NG(0)\sum_{k\in K_j}\abs{x_k}^2
=\int_{\RR^N}g(t)\abs{x_j(t)}^2\ dt
\le \alpha\int_{B_{R_j+r}} \abs{x_j(t)}^2\ dt.
\end{equation*}
It remains to show that $G(0)\ge \alpha r$. Using the variational characterization of the eigenvalue $\mu_{\gam_j/2}=R_j^2$ we have
\begin{align*}
G(0)
&=\int _{B_{\gam_j/2}}(R_j+r)^2H_{\gam_j/2}^2-\abs{\nabla H_{\gam_j/2}}^2\ dx\\
&=((R_j+r)^2-\mu_{\gam_j/2} )\int _{B_{\gam_j/2}} H_{\gam_j/2}^2\ dx\\
&=(2R_j+r)r\int _{B_{\gam/2}} H_{\gam_m/2}^2\ dx\\
&\ge \alpha r.\qedhere
\end{align*}
\end{proof}

In order to improve the last result we need a technical lemma, which generalizes to $N>1$ a well-known property of the function $\frac{\sin\omega}{\omega }$. We define
\begin{equation*}
g(\omega):=\frac{1}{V_1}\int_{B_1 } e^{i\omega\cdot s} \ ds=\frac{1}{V_1}\int_{B_1 } \cos(\omega\cdot s) \ ds,\quad \omega\in\RR^N
\end{equation*}
where $V_1$ denotes the volume of $B_1$. 

\begin{lemma}\label{l23}
We have
\begin{equation*}
\inf_{\abs{\omega}\ge t} 1-g(\omega) \ge \alpha_{m+1} t^2
\end{equation*}
for all $0\le t\le R_0\gam/(2m)$.
\end{lemma}

\begin{proof}
First we observe that $g(\omega)$ depends only on $\abs{\omega}$. Using Taylor's formula and assuming by rotation invariance that $\omega$ is parallel to the first coordinate axis, for $\omega\to 0$ we have
\begin{align*}
1-g(\omega)
&=\frac{1}{V_1}\int_{B_1 } 1-\cos(\omega\cdot s) \ ds\\
&=\frac{1}{2V_1}\int_{B_1 } \abs{\omega\cdot s}^2 \ ds+O(\abs{\omega}^4)\\
&= \frac{\abs{\omega}^2}{2V_1}\int_{B_1 } s_1^2 \ ds+O(\abs{\omega}^4)\\
&=\frac{\abs{\omega}^2}{2NV_1}\int_{B_1 } \abs{s}^2 \ ds+O(\abs{\omega}^4)\\
&= \frac{\abs{\omega}^2}{2N+4}+O(\abs{\omega}^4)
\end{align*}
because denoting the surface area of $B_1$ by $\beta$ we have
\begin{equation*}
V_1=\int_0^1\beta t^{N-1}\ dt=\frac{\beta }{N}
\quad\text{and}\quad
\int_{B_1 } \abs{s}^2 \ ds=\int_0^1\beta t^{N+1}\ dt=\frac{\beta }{N+2}.
\end{equation*}
There exists therefore a number $0<t_0<R_0\gam/(2m)$ such that
\begin{equation}\label{21}
1-g(\omega) \ge \alpha_{m+1} \abs{\omega}^2\quad\text{for all $\omega$ satisfying}\quad \abs{\omega}\le t_0.
\end{equation}

Next we show that $g(\omega)\to 0$ if $\abs{\omega}\to\infty$. Since the characteristic function of $B_1$ may be approximated in $L^1(\RR^N)$ by step functions, it suffices to show that
\begin{equation*}
\int_I e^{i\omega\cdot s} \ ds\to 0\quad\text{as}\quad\abs{\omega}\to\infty
\end{equation*}
for each fixed $N$-dimensional interval $I=[-a,a]^N$. (A translation of $I$ does not change the absolute value of the integral.) Using the inequality
\begin{equation*}
\max_j\abs{\omega_j}\ge \frac{\abs{\omega}}{\sqrt{N}}
\end{equation*}
we have
\begin{equation*}
\left\vert \int_I e^{i\omega\cdot s} \ ds\right\vert
=a^N\left\vert\prod_{j=1}^N\frac{\sin \omega_ja}{\omega_ja}\right\vert
\le a^N\frac{\sqrt{N}}{\abs{\omega}a}\to 0.
\end{equation*}

Since $g$ is continuous and $1-g(\omega)>0$ for all $\omega\ne 0$, it follows that
\begin{equation*}
\inf_{\abs{\omega}\ge t_0} 1-g(\omega) >0.
\end{equation*}
Therefore, by diminishing the constant $\alpha_{m+1}$ in \eqref{21} we may also assume that
\begin{equation}\label{22}
1-g(\omega) \ge \alpha_{m+1} \left( \frac{R_0\gamma}{2m}\right) ^2\quad\text{whenever}\quad \abs{\omega}\ge t_0.
\end{equation}

It follows from \eqref{21} and \eqref{22} that
\begin{equation*}
1-g(\omega) \ge \alpha_{m+1} t^2\quad\text{whenever}\quad \abs{t}\le\min\left\lbrace \abs{\omega}, \frac{R_0\gamma}{2m}\right\rbrace ,
\end{equation*}
and this is equivalent to the statement of the lemma.
\end{proof}

The following result is an adaptation of a method due to Haraux \cite{Har1989}.

\begin{lemma}\label{l24}
Add an arbitrary element $k'\in K$ to some $K_j$ and denote the enlarged set by $K_j'$.
Then we have
\begin{equation*}
\alpha_j' r^5\sum_{k\in K_j'}\abs{x_k}^2\le \int_{B_{R_j+2r}}\left\vert \sum_{k\in K_j'} x_k e^{i\omega_k\cdot t}\right\vert^2\ dt.
\end{equation*}
\end{lemma}

\begin{proof}
Writing
\begin{equation*}
x(t):=\sum_{k\in K_j'} x_k e^{i\omega_k\cdot t}
\end{equation*}
and introducing the function
\begin{equation*}
y(t):=x(t)-\frac{1}{V_1}\int_{B_1} e^{-ir\omega_{k'}\cdot s}x(t+rs)\ ds
\end{equation*}
we have
\begin{equation*}
y(t)=\sum_{k\in K_j}(1-g(r\omega_k-r\omega_{k'}))x_ke^{i\omega_k\cdot t}=:\sum_{k\in K_j}y_ke^{i\omega_k\cdot t}.
\end{equation*}
Since $\abs{r\omega_k-r\omega_{k'}}\ge r\gam$ for all $k\in K_j$, using Lemmas \ref{l22} and \ref{l23} we have
\begin{equation}\label{23}
\alpha_j\alpha_{m+1}^2\gam^4r^5\sum_{k\in K_j}\abs{x_k}^2\le \alpha_j r\sum_{k\in K_j }\abs{y_k}^2\le \int_{B_{R_j+r}}\abs{y(t)}^2\ dt.
\end{equation}
Furthermore,
\begin{align*}
\abs{y(t)}^2&\le 2\abs{x(t)}^2+2\Bigl |\frac{1}{V_1}\int_{B_1 } e^{-i\omega_{k'}\cdot s}x(t+rs)\ ds\Bigr |^2\\
&\le  2\abs{x(t)}^2+\frac{2}{V_1}\int_{B_1 }\abs{x(t+rs)}^2\ ds\\
&=2\abs{x(t)}^2+\frac{2}{V_r}\int_{t+B_r} \abs{x(s)}^2\ ds
\end{align*}
where $V_r$ denotes the volume of $B_r$, so that
\begin{align*}
\int_{B_{R_j+r}}\abs{y(t)}^2\ dt
&\le 2\int_{B_{R_j+r}} \abs{x(t)}^2\ dt
+\frac{2}{V_r }\int_{B_{R_j+r}}\int_{t+B_r}
\abs{x(s)}^2\ ds \ dt\\
&= 2\int_{B_{R_j+r}} \abs{x(t)}^2\ dt
+\frac{2}{V_r }\int_{B_{R_j+2r}}\int_{(s+B_r)\cap B_{R_j+r} }
\abs{x(s)}^2\ dt \ ds\\
&\le  2\int_{B_{R_j+r}} \abs{x(t)}^2\ dt
+ 2 \int_{B_{R_j+2r}} \abs{x(s)}^2  \ ds\\
&\le  4 \int_{B_{R_j+2r}} \abs{x(s)}^2  \ ds.
\end{align*}
Combining this with \eqref{23} we conclude that
\begin{equation}\label{24}
\alpha_j\alpha_{m+1}^2\gam^4r^5\sum_{k\in K_j}\abs{x_k}^2\le 4\int_{B_{R_j+2r}} \abs{x(t)}^2  \ dt.
\end{equation}

It remains to estimate $x_{k'}$. We have
\begin{align*}
V_{R_j+2r}\abs{x_{k'}}^2
&=\int_{B_{R_j+2r}}\Bigl\vert x_{k'}e^{i\omega_{k'}\cdot t}\Bigr\vert^2\ dt\\
&\le 2\int_{B_{R_j+2r}} \abs{x(t)}^2+\Bigl\vert \sum_{k\in K_j} x_k e^{i\omega_k\cdot t}\Bigr\vert^2\ dt\\
&\le 2\int_{B_{R_j+2r}} \abs{x(t)}^2\ dt + 2\alpha_0\sum_{k\in K_j}\abs{x_k}^2.
\end{align*}
Combining with \eqref{24} we get finally that
\begin{align*}
\alpha_j\alpha_{m+1}^2\gam^4r^5\sum_{k\in K_j'}\abs{x_k}^2
&\le \left(4+\frac{2\alpha_j\alpha_{m+1}^2\gam^4r^5}{V_{R_j+2r}}+\frac{8\alpha_0}{V_{R_j+2r}}\right) \int_{B_{R_j+2r}} \abs{x(t)}^2  \ dt\\
&\le \alpha\int_{B_{R_j+2r}} \abs{x(t)}^2  \ dt.\qedhere
\end{align*}
\end{proof}

We recall the following classical result on biorthogonal sequences.

\begin{lemma}\label{l25}
Let $(f_k)_{k\in K}$ be a family of vectors in a Hilbert space $H$. Assume that
\begin{equation*}
c_1'\sum_{k\in K}\abs{x_k}^2\le
\Bigl\Vert \sum_{k\in K} x_kf_k\Bigr\Vert ^2\le c_2'\sum_{k\in K}\abs{x_k}^2
\end{equation*}
for all finite linear combinations of these vectors, with two positive constants
$c_1'$ and $c_2'$. Then there exists another  family $(y_k)_{k\in K}$ of vectors in $H$
such that
\begin{equation*}
(y_k,f_n)=\delta_{kn}
\end{equation*}
for all $k,n\in K$. Moreover, both families are bounded in $H$:
\begin{equation*}
\norm{f_k}\le \sqrt{c_2'}\quad\text{and}\quad \norm{y_k}\le 1/\sqrt{c_1'}
\end{equation*}
for all $k$.
\end{lemma}

\begin{proof}
The formula
\begin{equation*}
y_n:=\norm{f_n-w_n}^{-2}(f_n-w_n)
\end{equation*}
where $w_n$ denotes the orthogonal projection of $f_n$ onto the closed linear subspace of $H$ spanned by the remaining vectors $f_k$, defines a family with the required properties.
\end{proof}

The proof of the theorem can now be completed by following Kahane's method \cite{Kah1962}. Applying Lemmas \ref{l22} and \ref{l25} we define for each $k\in K_j$ a function $\psi_{k,j}\in L^2(B_{R_j+r})$ satisfying
$\hat\psi_{k,j}(\omega_k)=1$, $\hat\psi_{k,j}(\omega_n)=0$ for all $n\in K_j\setminus\set{k}$, and (using H\"older's inequality)
\begin{equation*}
\norm{\psi_{k,j}}_{L^1(B_{R_j+r})}^2\le \frac{V_{R_j+r}}{\alpha_jr}\le \frac{\alpha }{r}.
\end{equation*}
Analogously, using Lemmas \ref{l24} and \ref{l25} we define for each $k\in K\setminus K_j$ a function $\psi_{k,j}\in L^2(B_{R_j+2r})$ such that
$\hat\psi_{k,j}(\omega_k)=1$, $\hat\psi_{k,j}(\omega_n)=0$ for all $n\in K_j$, and
\begin{equation*}
\norm{\psi_{k,j}}_{L^1(B_{R+j+2r})}^2\le \frac{V_{R_j+2r}}{\alpha_j'r^5}\le \frac{\alpha }{r^5}.
\end{equation*}
Setting
\begin{equation*}
\rho_k:=\psi_{k,1}*\cdots *\psi_{k,m}
\end{equation*}
we get $\rho_k\in L^2(B_{R-r})$ satisfying $\hat\rho_k(\omega_n)=\delta_{kn}$ for all $n,k\in K$, and
\begin{equation}\label{25}
\norm{\hat\rho_k}_{L^{\infty}(\RR^N)}^2\le \norm{\rho_k}_{L^1(B_{R-r})}^2\le \frac{\alpha}{r^{5m-4}}.
\end{equation}

Now we introduce the functions
\begin{equation*}
G:=H_{r/2}*H_{r/2},\quad g:=\abs{h_{r/2}}^2,\quad G_k(t):= G(t)e^{i\omega_kt},
\end{equation*}
and for every
\begin{equation*}
x(t)=\sum_{k\in K} x_k e^{i\omega_k\cdot t}
\end{equation*}
we define
\begin{equation*}
y(t)=(2\pi)^{-N}\sum_{k\in K} x_k (\rho_k*G_k)(t).
\end{equation*}
A straightforward computation and Plancherel's formula show that
\begin{equation}\label{26}
(y,x)_{L^2(B_{R})}=g(0)\sum_{k\in K}\abs{x_k}^2=h(0)^2\Bigl(\frac{r}{2}\Bigr)^{2N}\sum_{k\in K}\abs{x_k}^2
\end{equation}
and
\begin{equation*}
\norm{y}^2_{L^2(B_{R})}=(2\pi)^{-N}\int_{\RR^N}
\Bigl\vert\sum_{k\in K} x_k\widehat \rho_k(\omega)g_k(\omega)\Bigr\vert^2\ d\omega
\end{equation*}
where $g_k$ denotes the Fourier transform of $G_k$. Since $g_k\ge 0$, using \eqref{25} it follows that
\begin{equation*}
r^{5m-4}\norm{y}^2_{L^2(B_{R})}
\le \alpha\int_{\RR^N}
\Bigl\vert\sum_{k\in K} \abs{x_k} g_k(\omega)\Bigr\vert^2\ d\omega.
\end{equation*}
Using Plancherel's equality again, this is equivalent to the inequality
\begin{equation*}
r^{5m-4}\norm{y}_{L^2(B_{R})}^2
\le \alpha \int_{\RR^N}
\Bigl\vert G(t)\sum_{k\in K} \abs{x_k}  e^{i\omega_k\cdot t}\Bigr\vert^2\ dt.
\end{equation*}
Since $G(t)$ vanishes outside $B_{r}$ and
\begin{equation*}
\norm{G}_{L^\infty(\RR^N)}\le \norm{H_{r/2}}_{L^2(B_{r/2})}^2=\alpha r^N,
\end{equation*}
it follows that
\begin{equation*}
r^{5m-4}\norm{y}_{L^2(B_{R})}^2
\le \alpha r^{2N}\int_{B_{r}}
\Bigl\vert\sum_{k\in K} \abs{x_k} e^{i\omega_k\cdot t}\Bigr\vert^2\ dt.
\end{equation*}
Applying Lemma \ref{l21} we conclude that
\begin{equation}\label{27}
r^{5m-4}\norm{y}_{L^2(B_{R})}^2
\le \alpha r^{2N}\sum_{k\in K} \abs{x_k}^2.
\end{equation}
Combining \eqref{26} and \eqref{27} with the Cauchy--Schwarz inequality
\begin{equation*}
\left\vert (y,x)_{L^2(B_{R})}\right\vert^2\le \norm{y}_{L^2(B_{R})}^2\norm{x}_{L^2(B_{R})}^2
\end{equation*}
we obtain that
\begin{equation*}
r^{4N}\left( \sum_{k\in K} \abs{x_k}^2\right) ^2\le \alpha r^{2N+4-5m}\left( \sum_{k\in K} \abs{x_k}^2\right)\norm{x}_{L^2(B_{R})}^2
\end{equation*}
and hence
\begin{equation*}
r^{5m-4+2N}\sum_{k\in K}\abs{x_k}^2\le\alpha \norm{x}_{L^2(B_{R})}^2
\end{equation*}
as stated.

\remove{\emph{Questions.}

\begin{itemize}
\item Mehrenberger for $N>1$?

\item TenTuc  for $N>2$?
\end{itemize}}

\end{document}